\documentclass{amsart}
\usepackage{enumerate}
\usepackage[english]{babel}
\usepackage{amsmath,enumerate,graphicx}
\usepackage[misc,geometry]{ifsym}
\usepackage{amsthm}
\usepackage{amssymb}
\usepackage{amscd}

\author{Simone Virili}
\address{Universitat Aut\`onoma de Barcelona, Departament de Matem\`atiques,
Edifici C, Facultat de Ci\`encies
- 08193 - Bellaterra (Barcelona), Spain.}
\email{simone.virili@uab.cat \text{\rm or} virili.simone@gmail.com}

\usepackage[all]{xy}

\def\K{\mathbb F}
\def\Mat{\mathrm{Mat}}
\def\N{\mathbb N}
\def\C{\mathcal C}

\def\Z{\mathbb Z}

\def\End{\mathrm{End}}
\def\lmod#1{#1\text{-}{\mathrm{Mod}}}
\def\id{\mathrm{id}}
\def\Im{\mathrm{Im}}

\def\G{\mathcal{G}}

\def\T{\mathcal{T}}
\def\F{\mathcal{F}}
\def\S{\mathbf{S}}
\def\tor{\mathbf{T}}
\def\Q{\mathbf{Q}}

\def\Fun{\mathrm{Fun}}
\def\Aut{\mathrm{Aut}}
\def\bbone{\mathbf{1}}
\def\res{\downarrow}
\def\ext{\uparrow}
\def\Hom{\mathrm{Hom}}
\def\L{\mathcal L}

\makeatletter
\@namedef{subjclassname@2020}{%
  \textup{2020} Mathematics Subject Classification}
\makeatother

\theoremstyle{plain}
 \newtheorem{theorem}{Theorem}[section]
 \newtheorem{lemma}[theorem]{Lemma}
 \newtheorem{corollary}[theorem]{Corollary}
 \newtheorem{proposition}[theorem]{Proposition}

\theoremstyle{definition}

 \newtheorem{example}[theorem]{Example}

\begin{document}


\title{Stable finiteness of  endomorphism rings}


\keywords{Directly finite, stably finite, sofic, Kaplansky's conjecture, Grothendieck category.}
\subjclass[2020]{Primary 20C07, 18E15; secondary 16D10, 20C99, 18E35.}




\maketitle

\begin{abstract}
We combine a combinatorial idea of Benjamin Weiss and some localization theory of Grothendieck categories to give a short and completely self-contained proof of the following recent result of Hanfeng Li and Bingbing Liang: Given a left Noetherian ring $R$ and a sofic group $G$, the group-ring $R[G]$ is stably finite.
\end{abstract}


\section{Introduction}

A function between two sets is said to be {\em surjunctive} if it is non-injective or surjective. 
Let now $G$ be a group, $A$ a finite set and equip the product $A^G$ with the product of the discrete topologies on each copy of $A$ (so that $A^G$ is a compact space with a continuous $G$-action). A continuous and $G$-equivariant map $\phi\colon A^G\to A^G$ is also called a {\em cellular automaton} (on the finite alphabet $A$). A long standing open problem by Gottschalk \cite{Gott} is that of determining whether or not all cellular automata on finite alphabets are surjunctive; we refer to this problem as the {\em Surjunctivity Conjecture}. When $G$ is amenable this problem has been known for a long time to have a positive solution, but it was just in 1999 when Gromov \cite{Gro} came out with a general theorem solving the problem in the positive for the large class of sofic groups (see also \cite{Weiss}). The general case remains open.
%
%

\medskip
A related (or, actually, dual) problem is that of the stable finiteness of group rings (see \cite{Ceccherini} for connections between surjunctivity, ``linear surjunctivity'' and stable finiteness).
Indeed, a ring $R$ is {\em directly finite} if $xy=1$ implies $yx=1$ for all $x,y\in R$. Furthermore, $R$ is {\em stably finite} if the ring of square $k\times k$ matrices $\Mat_k(R)$ is directly finite for all $k\in \N_{\geq1}$. In other words, a ring $R$ is stably finite if, for each positive integer $k$, an endomorphism $\phi\colon R^k\to R^k$ is (split) surjective if and only if it is bijective.
A long standing open problem due to Kaplansky \cite{Kap} is to determine whether the group ring $\K[G]$ is stably finite for any field $\K$, we refer to this problem as the {\em Stable Finiteness Conjecture}.
%
In case the field $\K$  has characteristic $0$, then the problem was solved in the positive by Kaplansky. There was no progress in the positive characteristic case until the year 2002, when Ara, O'Meara and Perera \cite{Ara} proved that a group algebra $D[G]$ is stably finite whenever $G$ is residually amenable and $D$ is any division ring. This last result was generalized by Elek and Szab\'o \cite{Elek}, that proved the same result for $G$ a sofic group (see also \cite{Ceccherini} and \cite{Goulnara_L} for alternative proofs).

\medskip
Recently, the author \cite{Virili} proved that $R[G]$ is stably finite whenever $R$ is left Noetherian and $G$ is amenable using a notion of ``algebraic entropy'' for modules over $R[G]$. Soon after, Li and Liang \cite{Li-Liang} were able to extend this entropy theory to modules over $R[G]$ with $G$ sofic and, with this new invariant at hand, they could use similar ideas to \cite{Virili} to verify that $R[G]$ is stably finite whenever $R$ is left Noetherian and $G$ sofic. 
In this short note, we show that the main ideas of \cite{Weiss}, combined with some localization theory for Grothendieck categories, can be used to give a direct and completely self-contained proof of this last result. In fact, we obtain an even more general statement for representations of a sofic group $G$ on a Grothendieck category $\G$ (see Corollary \ref{gen_thm_groth3}). The announced result then follows just specializing to the case when $\G\cong R\text{-}\mathrm{Mod}$ for a left Noetherian ring $R$ (see Corollary \ref{gen_thm_groth4}).

\section{The main ingredients of the proof}

In this section we recall some basic definitions and facts that will be needed for the proofs of the main results. In particular, 
we start fixing some notations for labeled directed graphs and we recall a combinatorial result proved by Weiss (see Lemma~\ref{tech_weiss}): this simple-looking result is key for the proofs in Section~\ref{Sec_3}. After this, we introduce Cayley graphs and we use these to define the class of sofic groups. In the second part of the section, we focus instead on Grothendieck categories, on their localizations, and on the so-called Gabriel filtration of a Grothendieck category. Finally, in the third and last part of the section, we concentrate on the category of actions of a group on the objects of a given Grothendieck category.  

\subsection{Cayley Graphs and sofic groups}
A {\em labeled directed graph (digraph)} $\Gamma$ is a pair $(V,E)$, where $V$ is a set of {\em vertices} and $E\subseteq V\times V$ is a set of ordered pairs of vertices, called {\em edges}, together with a map
$\phi\colon E\to B$, from the set of edges to a set of {\em labelings} $B$. If we need to make explicit the set of labelings $B$ we say that $\Gamma$ is a $B$-labeled digraph.

Let $B$ be a set of labelings and consider two $B$-labeled digraphs $\Gamma_1=(V_1,E_1)$ and $\Gamma_2=(V_2,E_2)$, with labelings defined by $\phi_1\colon E_1\to B$ and $\phi_2\colon E_2\to B$, respectively. A {\em morphism of $B$-labeled digraphs} $\alpha\colon \Gamma_1\to \Gamma_2$ is a map $\alpha\colon V_1\to V_2$ that sends edges to edges (i.e., $(v,v')\in E_1$ implies $(\alpha(v),\alpha(v'))\in E_2$, for all $v,\,v'\in V_1$) and it respects the labels (i.e., $\phi_1(v,v')=\phi_2(\alpha(v),\alpha(v'))$, for all $(v,v')\in E_1$). In particular, such an $\alpha$ is an {\em isomorphism} if it is bijective and its inverse also sends edges to edges or, equivalently,  ($(\alpha(v),\alpha(v'))\in E_2$ implies $(v,v')\in E_1$, for all $v,\,v'\in V_1$).

Our main example of a labeled digraph is the following: given a finitely generated group $G$, fix a finite symmetric set of generators $B$ of $G$. The {\em Cayley graph} $\Gamma(G,B)$ of $G$ with respect to $B$ is a $B$-labeled digraph $\Gamma(G,B)=(V,E)$ such that the set of vertices $V$ coincides with $G$ and the edges are the pairs of the form $(g, gb)$, for $g\in G$ and $b\in B$; such an edge is labeled by $b$.

Given a (labeled) digraph $\Gamma=(V,E)$ we define the distance between vertices
\[
d\colon V\times V\longrightarrow \N\cup\{\infty\},
\]
where $d(v,v')$ is the number of edges in a shortest directed path between $v$ and $v'$, if some directed path exists; and $d(v,v')=\infty$ otherwise. For all $v\in V$ and $n\in\N$ we let 
\[
N_n(v):=\{v'\in V:d(v,v')\leq n\}
\] 
be the {\em $n$-th neighborhood of $v$ in $\Gamma$}.  In the case of a Cayley graph $\Gamma(G,B)$, the distance function is denoted by $d_B(-,-)$. For all $g\in G$ and $n\in \N$, we denote by $N_n(B,g)$ {\em $n$-th neighborhood of $g$ in $\Gamma(G,B)$}. If $g=1$ we usually denote $N_n(B,1)$ simply by $N_n(B)$. 

\begin{example}
Let $G=\Z$ be the additive group of integers, with the set of generators $B=\{-1,0,1\}$. Then $N_n(B)=\{-n,\dots,-1,0,1,\dots,n\}$. More generally, if we take 
$G=\Z^k=\Z e_1\times \dots\times \Z e_k$ for some positive integer $k$, the canonical choice for the set of generators is $B=\{-e_1,\dots,-e_k,0,e_1,\dots, e_k\}$. Thus we have 
\begin{equation}\label{palle_Z_k}N_n(B)=\left\{\sum_{i=1}^k\lambda_ie_i:\lambda_i\in \{-n,\dots,0,\dots,n\}\text{ and }\sum_{i=1}^k|\lambda_i|\leq n\right\}\, .\end{equation}
\end{example}

The following combinatorial lemma (which is Lemma 3.1 in \cite{Weiss}) will be extremely important later on.

\begin{lemma}\label{tech_weiss}
Let $G$ be a finitely generated group, $B$ a finite symmetric set of generators of $G$ and $\Gamma(G,B)$ its Cayley graph. If $\Lambda=(V, E)$ is a finite $B$-labeled digraph and, for at least half the vertices of $\Lambda$, their $2r_0+1$-neighborhoods in $\Lambda$ are isomorphic to $N_{2r_0+1}(B)$ (the $2r_0+1$-neighborhood of $1$ in $\Gamma(G,B)$), then there is a subset $V_1$ of these vertices such that
\begin{enumerate}[\rm (1)]
\item $|V_1|/|V| \geq 1/2|N_{2r_0+1}(B)|$; 
\item the minimal distance between any two vertices in $V_1$ is at least $2r_{0} + 1$.
\end{enumerate}
\end{lemma}

Let us conclude this subsection with the definition of sofic group. 
%
Indeed, a finitely generated group $G$ is {\em sofic} if, for each finite symmetric set of generators $B$ of $G$, $\varepsilon \in (0,1)$ and $r \in \N$, there is a finite $B$-labeled digraph $\Gamma=(V,E)$, which has a finite subset of its vertices $V_0 \subseteq V$ satisfying:
\begin{enumerate}[\rm (1)]
\item $|V_0| \geq (1 -\varepsilon )|V|$;
\item for each $v \in V_0$, the $r$-neighborhood $N_r(v)$ of $v$ in $\Gamma$ is isomorphic to $N_r(B)$.
\end{enumerate}
An arbitrary group $G$ is  {\em sofic} if each of its finitely generated subgroups is sofic.

\medskip
It is an easy exercise to prove that each finite and each Abelian group is sofic. In fact, one can even prove that the class of sofic groups contains both the class of amenable and the class of residually finite (or, more generally, that of residually amenable) groups. It is still not known whether any group is sofic (for more details see \cite{Weiss} and \cite{Elek2}).

\subsection{Grothendieck categories, localization and Gabriel dimension}

An Abelian category $\G$ is said to be a {\em Grothendieck} category if it is cocomplete, it has a generator, and directed colimits are exact in $\G$. One can prove that Grothendieck categories are also complete and that they have injective envelopes.

%

\medskip
A {\em torsion theory} in a Grothendieck category $\G$ is a pair $\tau=(\T, \F)$ of full subcategories of $\G$ such that 
\begin{itemize}
\item $\T\subseteq \G$ is closed under taking quotients, extensions and coproducts;
\item $\Hom_\G(T,F)=0$, for each $T\in \T$ and $F\in \F$;
\item for each $X\in \G$, there are $T_X\in \T$, $F_X\in \F$ and a short exact sequence:
\[
0\to T_X\to X\to F_X\to 0.
\]
\end{itemize}
A torsion theory $\tau=(\T,\F)$ is said to be {\em hereditary} provided $\T$ is closed under taking subobjects.
Given a torsion theory $\tau=(\T,\F)$ in a Grothendieck category $\G$, the inclusion $\T\to \G$ has a right adjoint $\tor_\tau\colon \G\to \T$, which is a subfunctor of the identity and it is  called the {\em torsion functor}. For each $X\in \G$, we have that $X/\tor_\tau X\in \F$.

\medskip
Recall now that a {\em Giraud subcategory} of a Grothendieck category $\G$, is a reflective subcategory $\C\subseteq \G$ such that the inclusion functor $\S\colon \C\to \G$ has an exact left adjoint $\Q\colon \G\to \C$. To such a Giraud subcategory, one can associate a unique hereditary torsion theory, whose torsion class is:
\[
\ker(\Q):=\{X\in \G:\Q(X)=0\}.
\]
%
%
%
%
On the other hand, given a hereditary torsion theory $\tau=(\T,\F)$ one can construct a Giraud subcategory of $\G$ as follows. We say that an object $X\in \G$ is {\em $\tau$-local} if both $X\in\F$ and $E(X)/X\in \F$. The full subcategory of $\G$ of all the $\tau$-local objects is denoted by $\G/\T$ and the inclusion functor by $\S_{\tau}\colon \G/\T\to \G$; one can give an explicit construction of the left adjoint $\Q_\tau\colon \G\to \G/\T$ (called the {\em $\tau$-quotient functor}) to $\S_\tau$, and prove that this functor is exact. Let us remark that we adopt the notation ``$\G/\T$'' because this category is equivalent to the Gabriel quotient of $\G$ over $\T$ (for more details on these constructions see \cite{sten}, \cite{gabriel} or \cite{popescu}). 
%
%


\medskip
Let us now recall that the {\em Gabriel filtration} of a Grothendieck category $\G$ is a transfinite chain 
$0=\G_{-1}\subseteq \G_0\subseteq \dots \subseteq \G_\alpha\subseteq \dots$
of hereditary torsion classes, where:
\begin{itemize}
\item $\G_{-1}:=0$;
\item suppose that $\alpha$ is an ordinal for which $\G_{\alpha}$ has already been defined. An object $X\in \G$ is said to be {\em $\alpha$-cocritical} if $X$ is $\tau_\alpha$-torsion free and every proper quotient of $X$ is $\tau_\alpha$-torsion, where $\tau_{\alpha}$ is the unique torsion theory whose torsion class is $\G_{{\alpha}}$. 
We let $\G_{\alpha+1}$ be the smallest hereditary torsion class containing $\G_{{\alpha}}$ and all the ${\alpha}$-cocritical objects;
\item if $\lambda$ is a limit ordinal, $\G_\lambda$ is the smallest hereditary torsion class containing $\bigcup_{\alpha<\lambda}\G_\alpha$.
\end{itemize}
For any ordinal $\alpha$, we let $\tor_\alpha\colon \G\to \G_{\alpha}$ and $\Q_\alpha\colon \G\to\G/\G_{\alpha}$ be the $\tau_\alpha$-torsion and the $\tau_\alpha$-quotient functor, respectively. Note that the ascending chain of subcategories $\G_\alpha\subseteq \G$ eventually stabilizes (one way to see this is to note that $\G_\alpha$ is generated by the $\tau_\alpha$-torsion quotients of a fixed generator $G$ of $\G$). Hence, it makes sense to consider $\widetilde\G:=\bigcup_{\alpha}\G_\alpha$. Let us remark that, even for $\G=\lmod R$ a category of modules, it may happen that $\widetilde \G\neq \G$ (e.g., take $R$ to be a non-discrete valuation domain). 

\begin{lemma}\label{detect_trivial_lemma}
Let $\G$ be a Grothendieck category, and $X\in \widetilde\G$. Then, $X=0$ if, and only if, $\Q_\alpha(\tor_{\alpha+1}(X))=0$, for all $\alpha$.
\end{lemma}
\begin{proof}
By definition of $\widetilde \G$, there is an ordinal $\delta\geq -1$ such that $X\in \G_\delta$. We proceed by transfinite induction on such a $\delta$. Indeed, if $\delta=-1$, then $X=0$ and so there is nothing to prove. If $\delta=\alpha+1$ (so that $X=\tor_{\alpha+1}(X)$), then there is a short exact sequence
\[
0\to \tor_{\alpha}(X)\to X\to X/\tor_{\alpha}(X)\to 0.
\]
Now, $\tor_\alpha(X)=0$ by inductive hypothesis, while $X/\tor_{\alpha}(X)\leq \Q_{\alpha}(\tor_{\alpha+1}(X))=0$, by hypothesis. Hence, $X=0$ as desired. Finally, suppose that $\delta$ is a limit ordinal. In this case, $X\cong \bigcup_{\alpha<\delta}\tor_{\alpha}(X)$ and, by inductive hypothesis, $\tor_{\alpha}(X)=0$ for each $\alpha<\delta$, showing that $X=0$.
\end{proof}

Recall that a Grothendieck category $\G$ is said to be {\em semi-Artinian} if, and only if, every non-trivial object $X\in \G$ has a simple (i.e., $(-1)$-cocritical) subobject. In particular, given an object $X$ in a semi-Artinian Grothendieck category $\G$, there exists an ordinal $\alpha$ and a continuous chain $\{X_\beta:\beta<\alpha\}$ of subobjects of $X$ such that $X=\bigcup_{\beta<\alpha}X_\beta$ and where $X_{\beta+1}/X_\beta$ is a simple object for all $\beta<\alpha$. 
Note that a Grothendieck category is semi-Artinian if, and only if, $\G=\G_0$.

\begin{lemma}\label{pre1}
Given a Grothendieck category $\G$, the category $\G_{\alpha+1}/\G_\alpha$ is semi-Artinian for each ordinal $\alpha$.
%
\end{lemma}
\begin{proof}
By construction, we have that $\G_{\alpha+1}/\G_\alpha\cong (\G/\G_{\alpha})_0$.
\end{proof}

Let $\G$ be a Grothendieck category, $X\in \G$, and denote by $\mathcal L(X)$ the lattice of subobjects of $X$. We say that $X$ is {\em Noetherian}, {\em Artinian} or of {\em finite length} if the lattice $\mathcal L(X)$ has the ascending chain condition, the descending chain condition or both, respectively. In case $X$ is of finite length, we denote by $\ell(X)\in \N$ the composition length of the lattice $\mathcal L(X)$.

\begin{lemma}\label{prop_gabriel_lemma}
Let $\G$ be a Grothendieck category, and $0\neq N\in \G$ a Noetherian object. Then, the following statements hold true:
\begin{enumerate}[\rm (1)]
\item $N\in \widetilde\G$;
\item $\tor_{\alpha}(N)$ is a Noetherian object in $\G_{\alpha}$, for each ordinal $\alpha$;
\item $\Q_{\alpha}(N)$ is a Noetherian object in $\G/\G_{\alpha}$, for each ordinal $\alpha$;
\item $\Q_\alpha(\tor_{\alpha+1}(N))$ is an object of finite length in $\G_{\alpha+1}/\G_{\alpha}$, for each ordinal $\alpha$.
\end{enumerate}
\end{lemma}
\begin{proof}
(1). Suppose, looking for a contradiction, that $N\notin \widetilde\G$. By Noetherianity, we can select a maximal subobject $K\leq N$ such that $N/K\notin \widetilde\G$. Then, for each $K\lneq H\leq N$ there exists an ordinal $\alpha_H$ such that $N/H\in \G_{\alpha_H}$, let $\alpha:=\sup_H\alpha_H$. By construction, $N/K$ is $\alpha$-cocritical, so $N/K\in \G_{\alpha+1}\subseteq \widetilde \G$, which is a contradiction.

\smallskip\noindent
(2). The lattice  $\L_{\G_\alpha}(\tor_\alpha(N))$ of subobjects of $\tor_\alpha(N)$ in $\G_\alpha$ can be identified with the sublattice of $\L(N)$ of those $K\in \L(N)$ such that $K\leq \tor_\alpha(N)$. In particular, any ascending chain in $\L_{\G_\alpha}(\tor_\alpha(N))$ induces an ascending chain in $\L(N)$ and, therefore, it stabilizes after a finite number of steps.

\smallskip\noindent
(3). The lattice $\L_{\G/\G_\alpha}(\Q_\alpha(N))$  of subobjects of $\Q_\alpha(N)$ in $\G/\G_\alpha$ can be identified with the sublattice of $\L(N)$ of those $K\in \L(N)$ such that $N/K$ is $\tau_\alpha$-torsion free (this is proved in \cite[Corollary IX.4.4]{sten} for categories of modules, and the argument carries over to general Grothendieck categories). In particular, any ascending chain in $\L_{\G/\G_\alpha}(\Q_\alpha(N))$ induces an ascending chain in $\L(N)$ and, therefore, it stabilizes after a finite number of steps.

\smallskip\noindent
(4). By (2) and (3) we deduce that $\Q_\alpha(\tor_{\alpha+1}(N))$ is a Noetherian object in $\G_{\alpha+1}/\G_\alpha$. As we have previously observed, $\G_{\alpha+1}/\G_\alpha$ is semi-Artinian and, therefore, $\Q_\alpha(\tor_{\alpha+1}(N))$ can be written as the union of a continuous chain whose subfactors are simple objects. By Noetherianity, any such ascending chain has to be finite. 
\end{proof}

\subsection{Group actions on objects of a Grothendieck category}

Let $\G$ be a Grothendieck category and  consider a group $G$ as a category with one object. The {\em category of representations} of $G$ on $\G$ is the functor category $\G^G:=\Fun(G,\G)$. An alternative, but equivalent, description of this category can be given as follows:
\begin{itemize}
\item each object of $\G^G$ can be thought of as a pair ${}_GX:=(X,\lambda\colon G\to \Aut_\G(X))$, with $X$ in $\G$ and $\lambda$ a homomorphism to the automorphism group of $X$;
\item given two objects ${}_GX_i:=(X_i,\lambda_i\colon G\to \Aut_\G(X_i))$ (with $i=1,2$), a morphism $\phi\colon {}_GX_1\to {}_GX_2$ in $\G^G$ is a $G$-{\em equivariant} morphism $\phi\colon X_1\to X_2$ in $\G$, that is, the following square commutes for all $g\in G$:
\[
\xymatrix{
X_1\ar[d]_{\lambda_1(g)}\ar[r]^{\phi}&X_2\ar[d]^{\lambda_2(g)}\\
X_1\ar[r]_\phi&X_2.}
\]
\end{itemize}
As an example, one can take the trivial group $G=\bbone$ and, in that case, it is easy to verify that $\G^{\bbone}\cong \G$. 
Consider a group $G$ and take the obvious inclusion $\iota\colon\bbone\to G$. This induces a {\em restriction functor} (also called a {\em forgetful} functor)
\[
\res_{\bbone}^G\colon \G^G\to \G,
\]
mapping a $G$-representation ${}_GX\mapsto X$. The left adjoint to $\res_{\bbone}^G$ is the {\em extension functor}
\[
\ext_{\bbone}^G\colon \G\to \G^{G}.
\]
Given an object $X\in \G$, we have that $\ext_{\bbone}^G(X)\cong(X^{(G)},\beta\colon G\to \Aut_{\G}(X^{(G)}))$, where $X^{(G)}$ is the coproduct of $|G|$-many copies of $X$ in $\G$, and  $\beta$ is the usual {\em Bernoulli shift} on $X^{(G)}$, that is, for each $g\in G$, $\beta^g\colon X^{(G)}\to X^{(G)}$ can be represented by the column-finite square matrix $\beta^g=(\beta^g_{g_2,g_1})_{g_1,g_2\in G}\in \Mat_{|G|}(\End_\G(X))$:
\[
\beta^g_{g_2,g_1}:=\begin{cases}\id_X&\text{if $g_2=gg_1$;}\\
0&\text{otherwise.}\end{cases}
\]
\begin{example}
Let $R$ be a ring and $G$ a group. The {\em group-ring} $R[G]$ is defined as the Abelian group $\bigoplus_{g\in G}\underline{g}R$, with the sum defined componentwise, and product defined as follows: given $x=\sum_{g\in G}\underline g x_g$ and $y=\sum_{g\in G}\underline g y_g \in R[G]$,
\[
x\cdot y:=\sum_{g\in G}\underline g\left(\sum_{h_1h_2=g}x_{h_1}y_{h_2}\right).
\]
With these operations, $R[G]$ becomes an associative and unitary ring. Consider now the category of left $R[G]$-modules $\lmod{R[G]}$. Essentially by definition, $\lmod{R[G]}$ is equivalent to the category of additive functors $\mathrm{Add}(R[G],\mathrm{Ab})$ from $R[G]$, viewed as a one-object preadditive category, to the category of Abelian groups $\mathrm{Ab}$. Similarly, these categories are both equivalent to $\mathrm{Add}(\Z[G],\lmod R)$ and also to the category of (non-necessarily additive) functors $\mathrm{Fun}(G,\lmod R)$, that is, $\lmod{R[G]}\cong (\lmod R)^G$. In this special example, the restriction and induction functors are called, respectively, the forgetful functor and the extension of scalars:
\[
\res_{\bbone}^G\cong \Hom_{R[G]}({}_{R[G]}R[G]_R,-)\qquad\text{and}\qquad \ext_{\bbone}^G\cong {}_{R[G]}R[G]_R\otimes_R-.
\]
\end{example}

\medskip
Take now a Grothendieck category $\G$ and a group $G$. Limits and colimits in $\G^G$ can be computed ``as in $\G$''. More explicitly, given a small category $I$ and a functor $F\colon I\to \G^G$, its (co)limit is the (co)limit of $\res^G_\bbone\circ F\colon I\to \G$, endowed with the unique compatible $G$-action. In particular, the fact that $\G$ is a cocomplete Abelian category with exact directed colimits implies that $\G^G$ has the same properties. Furthermore, given a generator $X\in \G$, the induced representation $\ext^G_\bbone X$ is a generator of $\G^G$, which is therefore a Grothendieck category.

\medskip
An object $X\in \G$ is said to be {\em compact} if the functor $\Hom_\G(X,-)\colon \G\to \mathrm{Ab}$ commutes with coproducts. In particular, if $X$ is Noetherian (or, more generally, finitely generated, that is, $\Hom_\G(X,-)$ commutes with monomorphic directed colimits) then it is compact. Observe that, given a compact object $X\in \G$, an object $Y\in \G$ and a set $I$, we have the following natural isomorphisms:
\[
\xymatrix{
\Hom_\G(X^{(I)},Y^{(I)})\cong \prod_I\Hom_\G(X,Y^{(I)})\cong \prod_I\bigoplus_I\Hom_\G(X,Y),
}
\]   
that is, a morphism $\phi\colon X^{(I)}\to Y^{(I)}$ can be naturally considered as a column-finite matrix with coefficients in $\Hom_\G(X,Y)$.

\begin{proposition}\label{explicit_G_covariance}
Let $\G$ be a Grothendieck category, $X,\, Y\in \G$ compact objects, and let $G$ be a group. Consider a morphism $\phi=(\phi_{g_2,g_1})_{g_1,g_2\in G}\colon X^{(G)}\to Y^{(G)}$ in $\G$. Then, $\phi$ represents a morphism $\ext^G_\bbone X \to \ext^G_\bbone Y$ in $\G^{G}$ (that is, it is $G$-equivariant with respect to the Bernoulli actions) if and only if:
\[
\phi_{g^{-1}g_2,g_1}=\phi_{g_2,gg_1},\qquad\text{for all $g,\, g_1,\, g_2\in G$.}
\]
Furthermore, given a morphism $\psi=(\psi_{g_2,g_1})_{g_1,g_2}\colon \ext_{\bbone}^G Y\to \ext^G_\bbone X $ in $\G^G$, so that $\psi_{g^{-1}g_2,g_1}=\psi_{g_2,gg_1}$ for all $g,\, g_1,\, g_2\in G$, we have that $\phi\circ\psi=\id_{\ext_{\bbone}^GY}$ if, and only if, the following equality is verified, for all $g_1,\, g_2\in G$:
\[
\sum_{h\in G}\phi_{g_2,h} \psi_{h,g_1}=\begin{cases}\id_Y&\text{if $g_1=g_2$;}\\
0&\text{otherwise.}
\end{cases}
\]
\end{proposition}
\begin{proof}
For the first claim about $\phi$, we have to verify that $\beta^g\phi=\phi\beta^g$ for all $g\in G$, and it is easily seen that, for each $g,\, g_1,\, g_2\in G$:
\begin{align*}
(\beta^g\phi)_{g_2,g_1}=\sum_{h\in G}\beta^g_{g_2,h}\phi_{h,g_1}=\phi_{g^{-1}g_2,g_1} \qquad\text{and}\\
 (\phi\beta^g)_{g_2,g_1}=\sum_{h\in G}\phi_{g_2,h}\beta^{g}_{h,g_1}=\phi_{g_2,gg_1}.
\end{align*}
It is also clear that $\phi\circ\psi=\id_{\ext_{\bbone}^GY}$ if, and only if, $(\phi\circ\psi)_{g_2,g_1}=\id_Y$ if $g_1=g_2$  and $(\phi\circ\psi)_{g_2,g_1}=0$ if $g_1\neq g_2\in G$. Furthermore, $(\phi\circ\psi)_{g_2,g_1}=\sum_{h\in G}\phi_{g_2,h}\psi_{h,g_1}$\end{proof}

Let us conclude with a brief discussion about the behavior of representations with respect to torsion and localization:

\begin{proposition}
Let $\G$ be a Grothendieck category, $\tau=(\T,\F)$ a hereditary torsion theory in $\G$ and $G$ a group. Define:
\begin{itemize}
\item $\T^G:=\{(X,\lambda\colon G\to \End_\G(X))\in \G^G:X\in\T\}$;
\item $\F^G:=\{(X,\lambda\colon G\to \End_\G(X))\in \G^G:X\in\F\}$;
\item $\tau^{G}:=(\T^G,\F^G)$.
\end{itemize}
Then, $\tau^G$ is a hereditary torsion theory in $\G^G$ and, given $(X,\lambda\colon G\to \End_\G(X))\in \G^G$,
\begin{enumerate}[\rm (1)]
\item $\tor_{\tau^G}(X,\lambda)=(\tor_\tau(G),\tor_\tau(\lambda))$, where $(\tor_\tau\lambda)_g:=\tor_\tau(\lambda_g)$, for all $g\in G$;
\item there is an equivalence $\G^G/\T^G\cong (\G/\T)^G$ and, identifying these categories, we have that $\Q_{\tau^G}(X,\lambda)=(\Q_\tau(G),\Q_\tau(\lambda))$, where $(\Q_\tau\lambda)_g:=\Q_\tau(\lambda_g)$, for all $g\in G$.
\end{enumerate}
\end{proposition}
\begin{proof}
Let $(X,\lambda)\in \G^G$ and consider $(\tor_\tau X, \tor_\tau\lambda)\in \T^G$ and $(X/\tor_\tau X, \bar\lambda)\in \F^G$, such that $(\tor_\tau\lambda)_g:=\tor_\tau(\lambda_g)$ and $\bar\lambda_g\colon X/\tor_\tau X\to X/\tor_\tau X$ is the map induced on the quotient (for each $g\in G$). It is easily seen that $\iota\colon \tor_\tau X\to X$ is $G$-equivariant with respect to $\tor_\tau(\lambda)$ and $\lambda$, and that $\pi\colon X\to X/\tor_\tau X$  is $G$-equivariant with respect to $\lambda$ and $\bar\lambda$. We then obtain the following short exact sequence in $\G^G$:
\[
\xymatrix{
0\ar[r]& (\tor_\tau X,\tor_\tau\lambda)\ar[r]^-{\iota}&(X,\lambda)\ar[r]^-\pi&(X/\tor_\tau X,\bar\lambda)\ar[r]&0.
}
\]
Furthermore, one can use the closure properties of $\T$ to show that $\T^G$ is a hereditary torsion class. Finally,  given ${}_GT\in \T^G$ and ${}_GF\in \F^G$, 
\[
\Hom_{\G^G}({}_GT,{}_GF)\subseteq \Hom_{\G}(T,F)=0,
\] 
proving that $\tau^G$ is a hereditary torsion theory in $\G^G$. Take now the  adjunction
\[
\Q_\tau : \G\leftrightarrows \G/\T : \S_\tau,
\]
identify $\G/\T$ with the essential image of the fully faithful functor $\S_\tau$ (which is then identified with the inclusion in $\G$). For each $X\in \G$, we denote by $\eta_X\colon X\to \Q_\tau X$ the $X$-component of the unit of the above adjunction.  Consider now two functors:
\[
\Q^G_\tau : \G^G\leftrightarrows (\G/\T)^G : \S^G_\tau,
\]
where $\Q^G_\tau(X,\lambda):=(\Q_\tau X,\Q_\tau \lambda)$ for each $(X,\lambda)\in \G^G$, with $(\Q_\tau \lambda)_g:=\Q_\tau (\lambda_g)$ for each $g\in G$, and with $\S^{G}_\tau$ the inclusion of the full subcategory $(\G/\T)^G $ of $\G^G$ (with the same identifications above). Let $(X,\lambda)\in \G^G$ and note that $\eta_X\colon X\to \Q_\tau X$ is $G$-equivariant (when taking the $G$-action $\Q_\tau \lambda$ on $\Q_\tau X$) by  naturality. Let now $(Y,\lambda')\in (\G/\T)^G$ and take a morphism $\phi\colon {}_GX\to {}_GY$ in $\G^G$. By the adjunction $\Q_\tau \dashv \S_\tau$, there exists a unique morphism $\bar\phi\colon\Q_\tau X\to Y$ in $\G$ such that $\bar \phi\circ \eta_X=\phi$. Let us show that  $\bar \phi$ is  $G$-equivariant. Indeed, let $g\in G$, then:
\begin{align*}
\lambda_g'\circ\bar \phi\circ \eta_X&=\lambda_g'\circ \phi &\text{since $\bar\phi\circ \eta_X=\phi$;}\\
&=\phi\circ\lambda_g&\text{since $\phi$ is $G$-equivariant;}\\
&=\bar\phi\circ \eta_X\circ\lambda_g&\text{since $\bar\phi\circ \eta_X=\phi$;}\\
&=\bar \phi\circ \Q_\tau \lambda_g\circ \eta_X&\text{since $\eta_X$ is $G$-equivariant.}
\end{align*}
Since $\eta_X$ is the unit and $Y\in \G/\T$, the equality $\lambda_g'\circ\bar \phi\circ \eta_X=\bar \phi\circ \Q_\tau \lambda_g\circ \eta_X\colon X\to Y$ implies that $\lambda_g'\circ\bar \phi=\bar \phi\circ \Q_\tau \lambda_g$, as desired.
It is also easy to show that $\Q_\tau^G$ is exact (using the exactness of $\Q_\tau$) and, therefore, $(\G/\T)^G$ becomes a Giraud subcategory of $\G^G$, associated with the hereditary torsion class 
\[
\ker(\Q_\tau^G)=\{{}_GX\in \G^G:\Q_\tau(X)=0\}=\T^G.
\] 
In particular, $(\G/\T)^G\cong \G^G/\T^G$.
\end{proof}

\begin{corollary}\label{coro_detect_0_action}
Let $\G$ be a Grothendieck category, $G$ a group, $N\in \G$ a Noetherian object and $M:=\ext^G_\bbone N\in \G^G$. Let $_{G}K\leq M$ be a subobject of $M$ in $\G^G$, then, $K=0$ if, and only if, $\Q^G_\alpha(\tor^G_{\alpha+1}(K))=0$, for all $\alpha$.
\end{corollary}
\begin{proof}
By the above proposition it is clear that $\Q^G_\alpha(\tor^G_{\alpha+1}(K))=0$, for all $\alpha$, if and only if $\Q_\alpha(\tor_{\alpha+1}(\res^G_\bbone K))=0$, for all $\alpha$. By Lemma \ref{detect_trivial_lemma}, this is equivalent to say that $\res^G_\bbone K=0$ in $\G$. Finally, note that $K=0$ in $\G^G$ if, and only if, $\res^G_\bbone K=0$ in $\G$.
\end{proof}

\section{Stable finiteness of endomorphism rings}\label{Sec_3}

\subsection{The semi-Artinian case}

This subsection contains the main application of Weiss' idea to representations of a sofic group $G$ on a Grothendieck category $\G$. The other results in this section will be proved by reducing to this case with a combination of suitable torsion and quotient functors.

\begin{proposition}\label{red_semi_art}
Let $\G$ be a Grothendieck category, $G$ a finitely generated sofic group, $X\in\G$ an object of finite length, and let ${}_GM:=\ext^G_\bbone X \in \G^G$. Then, $\End_{\G^G}({}_GM)$ is directly finite.
\end{proposition}
\begin{proof}
Let $\phi=(\phi_{g_2,g_1})_{g_1,g_2\in G}$ and $\psi=(\psi_{g_2,g_1})_{g_1,g_2\in G}\in \End_{\G^G}({}_GM)$ be two endomorphisms such that $\phi\circ\psi=\id_{M}$, and let us verify that $\phi$ is a monomorphism. 
 Indeed, we suppose, looking for a contradiction, that $\ker(\phi))\neq 0$, and we choose a sufficiently big $r_0\in \N$ and a sufficiently small $\varepsilon\in (0,1)$ in order to apply the definition of sofic group (Step 1 below) to obtain a statement about objects of finite length (Step 2), from which the contradiction can be seen (Step 3).


\smallskip\noindent
\underline{Step 1}. {\em The choice of $r_0$ and $\varepsilon$.} Consider the following finite subset of $G$:
\[
S:=\{1\}\cup\{g\in G:\phi_{g,1}\neq 0\}\cup\{g\in G:\psi_{g,1}\neq 0\}.
\]
We denote by $r_1$ the minimal positive integer such that $S\cdot S\subseteq N_{r_1}(B)$ (where $B$ is a fixed finite symmetric set of generators for $G$). For each subset $N\subseteq G$, identify $X^{(N)}$ with a sub-coproduct of $X^{(G)}$ and note that $X^{(G)}=\bigcup_{n\in\N} X^{N_n(B)}$ is an increasing direct union. Therefore, $\ker(\phi)=\bigcup_{n\in\N}(X^{N_n(B)}\cap \ker(\phi))$ and so there exists a minimal positive integer $r_2$ such that $X^{N_n(B)}\cap \ker(\phi)\neq 0$ for all $n\geq r_2$. We let
\begin{itemize}
\item $r_0:=\max\{r_1,r_2\}$;
\item and we choose $\varepsilon\in (0,1)$ such that $\varepsilon\lneq 1/(2\,\ell(X) |N_{2r_0+1}(B)|)$.
\end{itemize}
The reason for this specific choice of $\varepsilon$ will be clear at the end of the proof.
Note that, by the choice of $r_0$ and Proposition \ref{explicit_G_covariance}, we have the following relation:
\begin{equation}\label{comp=id}
\sum_{g\in N_{r_0}(B)}\phi_{gh,1} \psi_{1,g}=\sum_{g\in N_{r_0}(B)}\phi_{h,g^{-1}} \psi_{g^{-1},1}=\sum_{g\in G}\phi_{h,g} \psi_{g,1}=\begin{cases}
\id_X&\text{if $h=1$;}\\
0&\text{otherwise.}
\end{cases}
\end{equation}

\noindent
\underline{Step 2}. {\em Reduction to objects of finite length.} By definition of sofic group, there exists a finite $B$-labeled directed graph $\Gamma=(V,E)$, and a subset $V_0\subseteq V$ such that
\begin{itemize}
\item $|V_0|\geq (1-\varepsilon)|V|$;
\item there is an isomorphisms of labeled digraphs $N_{2r_0+1}(v)\cong N_{2r_0+1}(B)$, for all $v\in V_0$.
\end{itemize}
Consider the following two subsets of $V$:
\[
V':=\{v\in V:\text{ $N_{r_0}(v)\cong N_{r_0}(B)$}\}\qquad\text{and}\qquad V'':=\{v\in V':\text{ $N_{r_0}(v)\subseteq V'$}\}.
\]
Note that $V_0\subseteq V''\subseteq V'\subseteq V$. For each $v'\in V'$, let $\varepsilon_{v'}\colon N_{r_0}(v')\to N_{r_0}(B)$ 
be a fixed isomorphism of $B$-labeled digraphs and define two morphism
\[
\bar\phi:=(\bar\phi_{v,v'})_{V\times V'}\colon X^{(V')}\to X^{(V)}\qquad\text{and}\qquad\bar \psi:=(\bar\psi_{v',v''})_{V'\times V''}\colon X^{(V'')}\to X^{(V')},\quad\text{by}
\] 
\[
\bar\phi_{v,v'}:=\begin{cases}
\phi_{\varepsilon_{v'}(v),1}&\text{if $v\in N_{r_0}(v')$;}\\
0&\text{otherwise;}
\end{cases}
\quad\text{and}\quad
\bar\psi_{v',v''}:=\begin{cases}
\psi_{1,\varepsilon_{v'}(v'')}&\text{if $v''\in N_{r_0}(v')$;}\\
0&\text{otherwise.}
\end{cases}
\]
By \eqref{comp=id} and the definition of $\bar\phi$ and $\bar\psi$, we have that, given $v_1,\,v_2\in V''$, 
\begin{align*}
\sum_{v'\in V'}\bar\phi_{v_2,v'}\bar\psi_{v',v_1}&=\sum_{v'\in N_{r_0}(v_1)\cap N_{r_0}(v_2)}\phi_{\varepsilon_{v'}(v_2),1}\psi_{1,\varepsilon_{v'}(v_1)}=\begin{cases}\id_X&\text{if $v_1=v_2$;}\\
0&\text{otherwise.}\end{cases}
\end{align*}
In particular, $X^{(V_0)}\leq X^{(V'')}\leq \Im(\bar\phi)$.

\smallskip\noindent
\underline{Step 3}. {\em Where the contradiction comes.}
Taking the composition lengths, we obtain the following lower bound for $\ell(\Im(\bar \phi))$
$$(1-\varepsilon)|V|\ell(X)\leq |V_0|\ell (X)=\ell (X^{(V_0)})\leq \ell (\Im(\bar\phi))\, .$$
The contradiction will follow by showing that $\ell (\Im(\bar \phi))\lneq (1-\varepsilon)|V| \ell (X)$. Choose a subset $V_1\subseteq V''\subseteq V$ as in Lemma \ref{tech_weiss} (this is possible because of the choice of $\varepsilon\leq 1/2$ that implies that more than half of the elements of $V$ actually belong in $V_0$). By the definition of $\bar\phi$, for all $v\in V_1$, there is a commutative diagram of the form
$$
\xymatrix@R=20pt@C=45pt{
X^{(N_{r_0}(v))}\ar[r]^{\bar \phi}\ar[d]|-\cong& X^{(N_{2r_0}(v))}\ar[d]|-\cong\\
X^{(N_{r_0}(B))}\ar[r]^\phi& X^{(N_{2r_0}(B))}.
}$$
By the choice of $r_0$, the restriction of $\phi$ to $X^{(N_{r_0}(B))}$ is not a monomorphism, and so the restriction of $\bar\phi$ to $X^{(N_{r_0}(v))}$ cannot be monic for any $v\in V_1$. In particular,
\begin{equation}\label{-1}
\ell \left(\bar\phi\left(X^{(N_{r_0}(v))}\right)\right)\leq \ell \left(X^{(N_{r_0}(v))}\right)-1=|N_{r_0}\left(B\right)| \ell(X)-1.
\end{equation}
Let $N_{r_0}(V_1):=\bigcup_{v\in V_1}N_{r_0}(v)\subseteq V'$, so that $X^{(V')}\cong X^{(N_{r_0}(V_1))}\oplus X^{(V'\setminus N_{r_0}(V_1))}$. Thus,
\begin{align*}
\ell  \left( \Im \left( \bar\phi \right)  \right) &=\ell  \left( \bar\phi \left(X^{(V')}\right) \right) \leq \ell  \left( \bar\phi \left(X^{(N_{r_0}(V_1))}\right) \right)    +    \ell  \left( \bar\phi \left(X^{(V'\setminus N_{r_0}(V_1))}\right) \right)  \\
&\leq \sum_{v\in V_1}\ell \left( \bar\phi \left(X^{(N_{r_0}(v))}\right) \right) + \ell \left( X \right)  \left( |V'|-|V_1||N_{r_0} \left( B \right) | \right) \\
&\leq |V_1| \left( |N_{r_0} ( B ) |\ell(X ) -1 \right) +\ell (X)  \left( |V'|-|V_1||N_{r_0} \left( B \right) | \right) \\
&\leq \ell(X) |V|-|V|/ ( 2|N_{2r_0+1} ( B )|  ) \\
&\lneq |V| \ell (X)  \left( 1-\varepsilon \right),
\end{align*}
where the second line comes by the fact that $|V'\setminus N_{r_0}(V_1)|=|V'|-|V_1||N_{r_0}(B)|$, the third line is obtained applying \eqref{-1} and the fourth line follows by recalling that $|V'|\leq |V|$ and the estimate in Lemma \ref{tech_weiss}(1). The final strict inequality then comes from the choice of $\varepsilon$ and it is the contradiction we were looking for.
\end{proof}

\subsection{The general result for Grothendieck categories}

At this point, we know that, given a Grothendieck category $\G$, a finitely generated sofic group $G$, and an object $N\in \G$, the endomorphism ring of $\ext_\bbone^GN$ is directly finite as far as $N$ is of finite length. In the following corollary we show that, in fact, this last hypothesis can be weakened a lot, as it is enough to assume that $N$ is Noetherian.

\begin{corollary}\label{gen_thm_groth}
Let $\G$ be a Grothendieck category, $G$ a finitely generated sofic group, $N\in\G$ a Noetherian object and ${}_GM:=\ext^G_\bbone N$. Then, $\End_{\G^G}({}_GM)$ is directly finite.
\end{corollary}
\begin{proof}
Let $\phi$, $\psi\in\End_{\G^G}({}_GM)$ such that $\phi\circ \psi=\id_M$, 
let ${}_GK:=\ker(\phi)\in \G^G$ and suppose, looking for a contradiction, that $K\neq 0$. By Corollary~\ref{coro_detect_0_action}, there is an ordinal $\alpha$ such that $\bar K:=\Q^G_\alpha(\tor^G_{\alpha+1}(K))\neq 0$. Let now $\bar \G:=\G_{\alpha+1}/\G_\alpha$, that is semi-Artinian,  $\bar N:=\Q_{\alpha}(\tor_{\alpha+1}(N))\in \bar \G$, that is of finite length (being both Noetherian and semi-Artinian) and  $\bar M:=\Q_{\alpha}^G(\tor^G_{\alpha+1}(M))\cong \ext^G_{\bbone}\bar N\in \bar \G^G$. Consider also the two morphisms $\bar\phi:=\Q_{\alpha}^G(\tor^G_{\alpha+1}(\phi))$ and $\bar\psi:=\Q_{\alpha}^G(\tor^G_{\alpha+1}(\psi))\in \End_{\bar \G^G}({}_G\bar M)$ and note that ${}_G\bar K=\ker(\bar\phi)$ in $\bar \G$. Since $\bar \G$ is semi-Artinian, Proposition \ref{red_semi_art} tell us that $\bar K=0$, which is a contradiction.
\end{proof}

In the above corollary we have weakened the hypotheses on the base object $N$. In the following corollary we show that also the hypotheses on the acting group $G$ can be weakened, in fact, we can take $G$ to be an arbitrary sofic group.

\begin{corollary}\label{gen_thm_groth2}
Let $\G$ be a Grothendieck category, $G$ an arbitrary sofic group (not necessarily finitely generated), $N\in\G$ a Noetherian object and ${}_GM:=\ext^G_\bbone N$. Then, $\End_{\G^G}({}_GM)$ is directly finite.
\end{corollary}
\begin{proof}
Suppose that we have two maps $\phi,\, \psi\in \End_{\G^G}({}_GM)$ such that $\phi\circ\psi=\id_M$. We suppose, looking for a contradiction, that $\ker(\phi)\neq 0$. Since $\G$ is a Grothendieck category, $\ker(\phi)=\sum \{\ker(\phi)\cap N^{(F)}:F\subseteq G\text{ finite}\}$, so we can find a finite subset $F\subseteq G$ such that $\ker(\phi)\cap N^{(F)}\neq 0$. Consider the following finite subset of $G$:
\[
S:=F \cup\{g\in G:\phi_{g,1}\neq 0\}\cup\{g\in G:\psi_{g,1}\neq 0\},
\]
and let $H\leq G$ be the subgroup generated by $S$, which is clearly finitely generated and sofic. We claim that $N^{(H)}$ is both a $\phi$- and a $\psi$-invariant subobject of $M$; to see this, consider $h\in H$, $g\in G$, and suppose that $\phi_{g,h}\neq 0$. Since $\phi$ is $G$-equivariant, $\phi_{g,h}=\phi_{h^{-1}g,1}$, so $h^{-1}g\in S\subseteq H$ that, together with the fact that $h\in H$, implies that $g=h(h^{-1}g)\in H$; the argument for $\psi$ is completely analogous. We can then consider $M':=\ext_{\bbone}^HN$ and the two maps $\phi':=\phi_{\restriction M'}$ and $\psi':=\psi_{\restriction M'}$, that are clearly $H$-equivariant, that is, they are morphisms in $\G^H$. Finally, note that $\phi'\circ\psi'=\id_{M'}$ and $0\neq \ker(\phi)\cap N^{(F)}\leq\ker(\phi')$, and this is absurd by Corollary \ref{gen_thm_groth}.
%
\end{proof}

Finally, let us show that the above statements about direct finiteness can be easily reformulated in terms of stable finiteness.

\begin{corollary}\label{gen_thm_groth3}
Let $\G$ be a Grothendieck category, $G$ a sofic group, $N\in\G$ a Noetherian object and ${}_GM:=\ext^G_\bbone N$. Then, $R:=\End_{\G^G}({}_GM)$ is stably finite.
\end{corollary}
\begin{proof}
Let $k\in \N_{>0}$ and consider the following isomorphism: 
\[
\Mat_k(R)\cong \End_{\G^G}(\ext_\bbone^G(N^{(k)})),
\] 
where $N^k$ is still a Noetherian object. Hence, to verify that $R$ is stably finite, that is, that $\Mat_k(R)$ is directly finite for each $k\in\N_{>0}$, it is enough to verify that $\End_{\G^G}(\ext_\bbone^GX)$ is directly finite for each Noetherian object $X\in \G$, and we already know this by Corollary \ref{gen_thm_groth2}.
\end{proof}

To conclude, let us  reformulate the above results in the special case when the category $\G=R\text{-}\mathrm{Mod}$ is the category of left $R$-modules for some ring $R$.

\begin{corollary}\label{gen_thm_groth4}
Let $R$ be a ring, $G$ a sofic group and ${}_RN$ a Noetherian left $R$-module. Then, $\End_{R[G]}(R[G]\otimes_RN)$ is a stably finite ring.
In particular, if $R$ is a left Noetherian ring, then $R[G]\cong\End_{R[G]}(R[G]\otimes_R R)$ is stably finite.
\end{corollary}




\bibliographystyle{emsplain}

\end{document}